\documentclass[10pt]{article}

\author{Gus Wiseman\footnote{Research supported in part by NSF VIGRE Grant No. DMS-0135345.}\\
\small Department of Mathematics,
\small University of California,
\small One Shields Ave.,
\small Davis, CA 95616\\
\small Email: \texttt{gus@math.ucdavis.edu}}
\date{Version of September 3, 2007}

\title{Enumeration of paths and cycles and $e$-coefficients of incomparability graphs}
\usepackage{amsfonts}
\usepackage{amsmath}
\usepackage{amsthm}
\usepackage{amssymb}
\usepackage[xcolor=pst]{pstricks}

\theoremstyle{plain}
\newtheorem{thm}{Theorem}[section]
\newtheorem{prop}[thm]{Proposition}
\newtheorem{lem}[thm]{Lemma}
\newtheorem{cor}[thm]{Corollary}

\theoremstyle{definition}
\newtheorem{defn}[thm]{Definition}

\theoremstyle{remark}
\newtheorem{ex}[thm]{Example}
\newtheorem{note}[thm]{Note}

\begin{document}

\bibliographystyle{hamsplain}
\maketitle

\begin{abstract}
We prove that the number of Hamiltonian paths on the complement of an acyclic digraph is equal to the number of cycle covers. As an application, we obtain a new expansion of the chromatic symmetric function of incomparability graphs in terms of elementary symmetric functions. Analysis of some of the combinatorial implications of this expansion leads to three bijections involving acyclic orientations.
\end{abstract}

\section{Introduction}

\begin{figure}
\psset{unit=0.625cm}
\begin{pspicture}(-0.5,-4.5)(20.5,1.5)
\psset{dotsize=0.18}
\psset{linewidth=0.05}

\psdots(0.0,0.0)
\psdots(0.0,1.0)
\psdots(1.0,0.0)
\psdots(1.0,1.0)

\psset{arrowscale=1.5 0.7}

\psset{arrowscale=1.2 1.6}
\psline{->}(0.0,0.0)(1.0,0.0)
\psline{->}(0.0,0.0)(0.0,1.0)
\psline{->}(0.0,1.0)(1.0,1.0)
\psline{->}(1.0,0.0)(0.0,1.0)

\uput{0.18}[45](1,1){\small{4}}
\uput{0.18}[-45](1,0){\small{2}}
\uput{0.18}[135](0,1){\small{3}}
\uput{0.18}[-135](0,0){\small{1}}

\psset{origin={-2.5,0.0}}

\psdots(0.0,0.0)
\psdots(0.0,1.0)
\psdots(1.0,0.0)
\psdots(1.0,1.0)

\psset{arrowscale=1.5 0.7}
\psarc{->}(-0.127,-0.127){0.18}{45}{405}
\psarc{->}(-0.127,1.127){0.18}{-45}{315}
\psarc{->}(1.127,1.127){0.18}{225}{585}
\psarc{->}(1.127,-0.127){0.18}{135}{495}

\psset{arrowscale=1.2 1.6}
\psline{<->}(0.0,0.0)(1.0,1.0)
\psline{<->}(1.0,0.0)(1.0,1.0)
\psline{->}(0.0,1.0)(0.0,0.0)
\psline{->}(0.0,1.0)(1.0,0.0)
\psline{->}(1.0,1.0)(1.0,0.0)
\psline{->}(1.0,0.0)(0.0,0.0)
\psline{->}(1.0,1.0)(0.0,1.0)

\psset{origin={-0.0,2.0}}

\psdots(0.0,0.0)
\psdots(0.0,1.0)
\psdots(1.0,0.0)
\psdots(1.0,1.0)

\psset{arrowscale=1.2 1.6}
\psline{<-}(0.0,0.0)(1.0,1.0)
\psline{->}(1.0,0.0)(1.0,1.0)
\psline{->}(0.0,1.0)(1.0,0.0)

\psset{origin={-2.5,2.0}}

\psdots(0.0,0.0)
\psdots(0.0,1.0)
\psdots(1.0,0.0)
\psdots(1.0,1.0)

\psset{arrowscale=1.2 1.6}
\psline{->}(0.0,0.0)(1.0,1.0)
\psline{->}(0.0,1.0)(1.0,0.0)
\psline{->}(1.0,0.0)(0.0,0.0)

\psset{origin={-5.0,2.0}}

\psdots(0.0,0.0)
\psdots(0.0,1.0)
\psdots(1.0,0.0)
\psdots(1.0,1.0)

\psset{arrowscale=1.2 1.6}

\psline{->}(0.0,0.0)(1.0,1.0)
\psline{->}(0.0,1.0)(0.0,0.0)
\psline{->}(1.0,1.0)(1.0,0.0)

\psset{origin={-7.5,2.0}}

\psdots(0.0,0.0)
\psdots(0.0,1.0)
\psdots(1.0,0.0)
\psdots(1.0,1.0)

\psset{arrowscale=1.2 1.6}
\psline{->}(0.0,1.0)(1.0,0.0)
\psline{->}(1.0,0.0)(0.0,0.0)
\psline{->}(1.0,1.0)(0.0,1.0)

\psset{origin={-10.0,2.0}}

\psdots(0.0,0.0)
\psdots(0.0,1.0)
\psdots(1.0,0.0)
\psdots(1.0,1.0)

\psset{arrowscale=1.2 1.6}
\psline{->}(1.0,0.0)(1.0,1.0)
\psline{->}(0.0,1.0)(0.0,0.0)
\psline{->}(1.0,1.0)(0.0,1.0)

\psset{origin={-12.5,2.0}}

\psdots(0.0,0.0)
\psdots(0.0,1.0)
\psdots(1.0,0.0)
\psdots(1.0,1.0)

\psset{arrowscale=1.2 1.6}
\psline{->}(0.0,0.0)(1.0,1.0)
\psline{->}(1.0,0.0)(0.0,0.0)
\psline{->}(1.0,1.0)(0.0,1.0)

\psset{origin={-15.0,2.0}}

\psdots(0.0,0.0)
\psdots(0.0,1.0)
\psdots(1.0,0.0)
\psdots(1.0,1.0)

\psset{arrowscale=1.2 1.6}
\psline{->}(0.0,0.0)(1.0,1.0)
\psline{->}(0.0,1.0)(1.0,0.0)
\psline{->}(1.0,1.0)(0.0,1.0)

\psset{origin={-0.0,4.0}}

\psdots(0.0,0.0)
\psdots(0.0,1.0)
\psdots(1.0,0.0)
\psdots(1.0,1.0)

\psset{arrowscale=1.5 0.7}
\psarc{->}(-0.127,-0.127){0.18}{45}{405}

\psset{arrowscale=1.2 1.6}
\psline{->}(1.0,1.0)(0.0,1.0)
\psline{->}(0.0,1.0)(1.0,0.0)
\psline{->}(1.0,0.0)(1.0,1.0)

\psset{origin={-2.5,4.0}}

\psdots(0.0,0.0)
\psdots(0.0,1.0)
\psdots(1.0,0.0)
\psdots(1.0,1.0)

\psset{arrowscale=1.5 0.7}

\psset{arrowscale=1.2 1.6}
\psline{->}(0.0,0.0)(1.0,1.0)
\psline{->}(1.0,1.0)(0.0,1.0)
\psline{->}(0.0,1.0)(1.0,0.0)
\psline{->}(1.0,0.0)(0.0,0.0)

\psset{origin={-5.0,4.0}}

\psdots(0.0,0.0)
\psdots(0.0,1.0)
\psdots(1.0,0.0)
\psdots(1.0,1.0)

\psset{arrowscale=1.5 0.7}
\psarc{->}(1.127,-0.127){0.18}{135}{495}

\psset{arrowscale=1.2 1.6}
\psline{->}(0.0,0.0)(1.0,1.0)
\psline{->}(1.0,1.0)(0.0,1.0)
\psline{->}(0.0,1.0)(0.0,0.0)

\psset{origin={-7.5,4.0}}

\psdots(0.0,0.0)
\psdots(0.0,1.0)
\psdots(1.0,0.0)
\psdots(1.0,1.0)

\psset{arrowscale=1.5 0.7}
\psarc{->}(-0.127,-0.127){0.18}{45}{405}
\psarc{->}(-0.127,1.127){0.18}{-45}{315}
\psarc{->}(1.127,1.127){0.18}{225}{585}
\psarc{->}(1.127,-0.127){0.18}{135}{495}

\psset{arrowscale=1.2 1.6}

\psset{origin={-10.0,4.0}}

\psdots(0.0,0.0)
\psdots(0.0,1.0)
\psdots(1.0,0.0)
\psdots(1.0,1.0)

\psset{arrowscale=1.5 0.7}
\psarc{->}(-0.127,-0.127){0.18}{45}{405}
\psarc{->}(-0.127,1.127){0.18}{-45}{315}

\psset{arrowscale=1.2 1.6}
\psline{<->}(1.0,0.0)(1.0,1.0)

\psset{origin={-12.5,4.0}}

\psdots(0.0,0.0)
\psdots(0.0,1.0)
\psdots(1.0,0.0)
\psdots(1.0,1.0)

\psset{arrowscale=1.5 0.7}
\psarc{->}(-0.127,1.127){0.18}{-45}{315}

\psset{arrowscale=1.2 1.6}
\psline{->}(0.0,0.0)(1.0,1.0)
\psline{->}(1.0,1.0)(1.0,0.0)
\psline{->}(1.0,0.0)(0.0,0.0)

\psset{origin={-15.0,4.0}}

\psdots(0.0,0.0)
\psdots(0.0,1.0)
\psdots(1.0,0.0)
\psdots(1.0,1.0)

\psset{arrowscale=1.5 0.7}
\psarc{->}(-0.127,1.127){0.18}{-45}{315}
\psarc{->}(1.127,-0.127){0.18}{135}{495}

\psset{arrowscale=1.2 1.6}
\psline{<->}(0.0,0.0)(1.0,1.0)

\end{pspicture}
\caption{An acyclic digraph $H$ and its complement $\overline{H}$ (first row), the seven Hamiltonian paths on $\overline{H}$ (second row), and the seven cycle covers of $\overline{H}$ (third row).}\label{bijfig}
\end{figure}
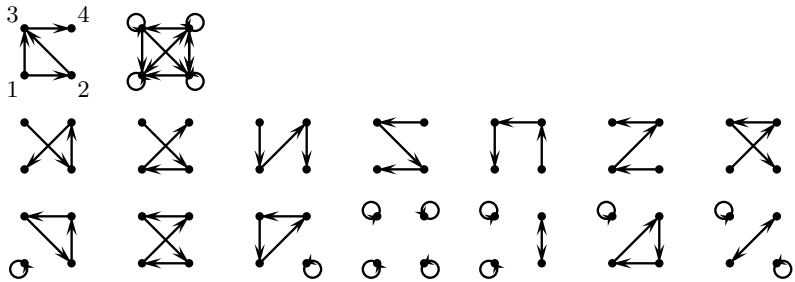

One of the most useful results in combinatorics is that a permutation can be represented either as a word or as a collection of cycles. Graphically, this may be regarded as a correspondence between Hamiltonian paths and cycle covers. The first result of this paper (Theorem \ref{thm:peqc2}) is that such a correspondence exists even if the arrows of the Hamiltonian paths and cycle covers are restricted by an acyclic digraph (see Figure \ref{bijfig}). We give both a simple inductive proof and a direct bijective proof. This particular fact does not seem to have been pointed out before, although it follows immediately from the following known result.

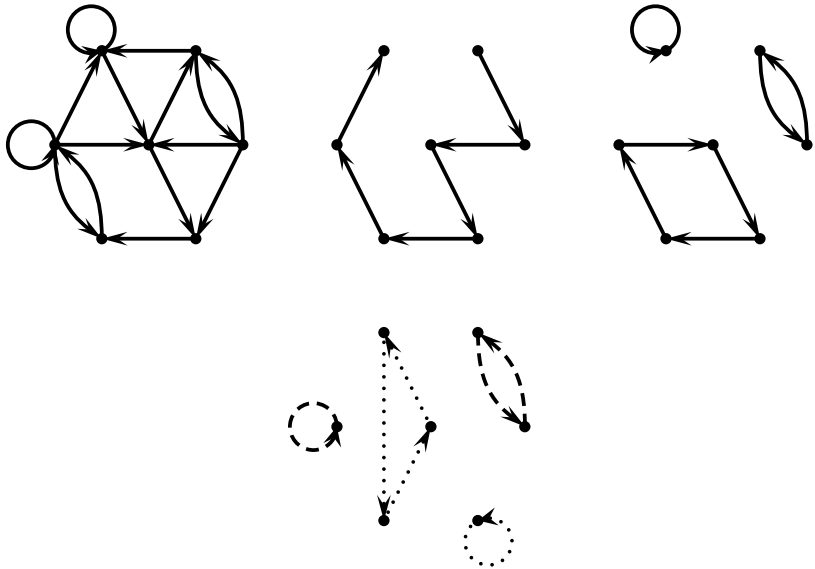
\begin{figure}
\psset{unit=2.5cm}
\begin{pspicture}(0.5,0.25)(5.0,3.25)
\psset{dotsize=0.06}
\psset{linewidth=0.02}
\psset{arrowscale=1.6 1.2}

\psset{origin={0.0,-1.5}}

\psdots(1.0,0.5)
\psdots(0.75,1.0)
\psdots(1.0,1.5)
\psdots(1.5,1.5)
\psdots(1.75,1.0)
\psdots(1.5,0.5)
\psdots(1.25,1.0)

\psarc{->}(0.625,1.0){0.125}{0}{360}
\psarc{->}(0.944,1.612){0.125}{-63.435}{296.565}

\psset{arrowscale=1.2 1.6}
\psline{->}(0.75,1.0)(1.0,1.5)
\psline{->}(1.5,1.5)(1.0,1.5)
\psline{->}(1.0,1.5)(1.25,1.0)
\psline{->}(0.75,1.0)(1.25,1.0)
\psline{->}(1.25,1.0)(1.5,1.5)
\psline{->}(1.25,1.0)(1.5,0.5)
\psline{->}(1.5,0.5)(1.0,0.5)
\psline{->}(1.75,1.0)(1.5,0.5)
\psline{->}(1.75,1.0)(1.25,1.0)

\psbezier{->}(1.5,1.5)(1.5,1.2765)(1.55,1.15)(1.75,1.0)
\psbezier{->}(1.75,1.0)(1.75,1.2235)(1.7,1.35)(1.5,1.5)

\psbezier{->}(0.75,1.0)(0.75,0.7765)(0.8,0.65)(1.0,0.5)
\psbezier{->}(1.0,0.5)(1.0,0.7235)(0.95,0.85)(0.75,1.0)

\psset{origin={-1.5,-1.5}}

\psset{dotsize=0.06}
\psset{linewidth=0.02}
\psset{arrowscale=1.6 1.2}

\psdots(1.0,0.5)
\psdots(0.75,1.0)
\psdots(1.0,1.5)
\psdots(1.5,1.5)
\psdots(1.75,1.0)
\psdots(1.5,0.5)
\psdots(1.25,1.0)

\psset{arrowscale=1.2 1.6}
\psline{->}(0.75,1.0)(1.0,1.5)
\psline{->}(1.25,1.0)(1.5,0.5)
\psline{->}(1.5,0.5)(1.0,0.5)
\psline{->}(1.75,1.0)(1.25,1.0)
\psline{->}(1.5,1.5)(1.75,1.0)
\psline{->}(1.0,0.5)(0.75,1.0)

\psset{origin={-3.0,-1.5}}

\psset{dotsize=0.06}
\psset{linewidth=0.02}
\psset{arrowscale=1.6 1.2}

\psdots(1.0,0.5)
\psdots(0.75,1.0)
\psdots(1.0,1.5)
\psdots(1.5,1.5)
\psdots(1.75,1.0)
\psdots(1.5,0.5)
\psdots(1.25,1.0)

\psarc{->}(0.944,1.612){0.125}{-63.435}{296.565}

\psset{arrowscale=1.2 1.6}
\psline{->}(1.25,1.0)(1.5,0.5)
\psline{->}(1.5,0.5)(1.0,0.5)
\psline{->}(0.75,1.0)(1.25,1.0)
\psline{->}(1.0,0.5)(0.75,1.0)

\psbezier{->}(1.5,1.5)(1.5,1.2765)(1.55,1.15)(1.75,1.0)
\psbezier{->}(1.75,1.0)(1.75,1.2235)(1.7,1.35)(1.5,1.5)

\psset{origin={-1.5,0.0}}

\psset{dotsize=0.06}
\psset{linewidth=0.02}
\psset{arrowscale=1.6 1.2}

\psdots(1.0,0.5)
\psdots(0.75,1.0)
\psdots(1.0,1.5)
\psdots(1.5,1.5)
\psdots(1.75,1.0)
\psdots(1.5,0.5)
\psdots(1.25,1.0)

\psset{linestyle=dashed}
\psarc{->}(0.625,1.0){0.125}{0}{360}
\psset{linestyle=dotted}
\psarc{->}(1.556,0.388){0.125}{116.565}{476.565}

\psset{arrowscale=1.2 1.6}
\psset{linestyle=dotted}
\psline{->}(1.0,1.5)(1.0,0.5)
\psline{->}(1.0,0.5)(1.25,1.0)
\psline{->}(1.25,1.0)(1.0,1.5)

\psset{linestyle=dashed}
\psbezier{->}(1.5,1.5)(1.5,1.2765)(1.55,1.15)(1.75,1.0)
\psbezier{->}(1.75,1.0)(1.75,1.2235)(1.7,1.35)(1.5,1.5)

\end{pspicture}
\caption{A digraph $H$, a Hamiltonian path on $H$, and a cycle cover of $H$ (first row); a pair of cycle covers of spanning induced subgraphs, the first (dashed lines) of $H$ and the second (dotted lines) of $\overline{H}$ (second row).}\label{exfig}
\end{figure}

Let $H$ be any digraph with vertex set $V$ (loops are allowed, but not multiple edges). The complement $\overline{H}$ is the digraph with an edge from $v$ to $w$ if and only if $H$ does not contain such an edge. If $T\subseteq V$ the subgraph $H|_T$ of $H$ \emph{induced} by $T$ has vertex set $T$ and edge set consisting of all edges of $H$ with both ends in $T$. The following sum will be taken over pairs of induced subgraphs; we say that such a pair is \emph{spanning} if the subgraphs are induced by a pair of disjoint sets with union $V$. Let $\pi_H$ be the number of Hamiltonian paths on $H$. A cycle cover is a subgraph of $H$ that is a union of disjoint cycles and such that every vertex is contained in some cycle. The length $\ell(F)$ of a cycle cover $F$ is the number of cycles; the weight $W(F)$ is the number of vertices. See Figure \ref{exfig} for examples of a Hamiltonian path and a cycle cover of a digraph. The following identity expresses the enumerative relationship between paths and cycle covers. It is proved in \cite[Theorem 4.1]{MR1928099} by Lass, who attributes it to Berge.
\begin{equation}\label{eq:pcgen}
\pi_H=\sum_{F_1, F_2}(-1)^{W(F_2)-\ell(F_2)},
\end{equation}
where the sum is over all pairs of cycle covers of spanning induced subgraphs, the first of $H$ and the second of $\overline{H}$. See Figure \ref{exfig} for an example of such a pair.

In the case where $H$ is the complement of an acyclic digraph there are no cycle covers of $\overline{H}$, so the right hand side enumerates cycle covers of $H$. Hence this special case reduces to our Theorem \ref{thm:peqc2}.

Using Theorem \ref{thm:peqc2} and a combinatorial lemma in symmetric function theory (Lemma \ref{lem:eco}), we obtain the following application. The chromatic symmetric function $X_G(x_1,x_2,\ldots)\stackrel{,}{=} X_G$ of a graph $G$ is a symmetric function that generalizes the chromatic polynomial. It was first studied in detail by Stanley \cite{MR96b:05174}, but the open problem that has been the motivation for much subsequent work, the \emph{poset-chain} conjecture, was posed earlier by Stanley and Stembridge~\cite[Conjecture 5.5]{MR1207737}. This conjecture asserts that the chromatic symmetric function of the incomparability graph of certain posets is $e$-positive, meaning the coefficients in its expansion in terms of elementary symmetric functions (called the $e$-coefficients of the graph) are all nonnegative. Our application (Corollary \ref{cor:eco}) is a new formula for the $e$-coefficients of incomparability graphs.

The simplest case of this formula exposes some interesting combinatorics. We describe two bijections involving acyclic orientations: the shatter bijection between weakly decreasing Hamiltonian paths and acyclic orientations of an incomparability graph (Theorem \ref{patho}) and the second-sink bijection, which for connected incomparability graphs gives a correspondence between circular acyclic orientations with a fixed greatest sink and acyclic orientations with a unique sink at the same fixed vertex (Proposition \ref{prop:sec}). We also elaborate the well-known result that the linear term of the chromatic polynomial is, up to sign, the number of acyclic orientations with a unique sink at a fixed vertex; this involves a third bijection, well known in trace theory (see \cite{MR1075995}), between acyclic orientations and stable link sequences (Theorem \ref{thm:sls}).


Throughout this article a comma placed over a binary relation indicates the presence of a relative clause. For example, the notation $S\stackrel{,}{\subseteq}V$ can be read as the noun phrase ``$S$, which is a subset of $V$.''

\section{Enumeration of paths and cycles}

A digraph $D$ consists of a vertex set $V$ and an edge or arrow set which is a subset of $V\times V$. Hence our digraphs are permitted to have loops but not multiple arrows. An arrow $(v,w)$ is said to point to $w$ and to point from $v$. A digraph $F$ whose vertex and arrow sets are subsets of the vertex and arrow sets of $D$ is a \emph{subdigraph} of $D$. If $F$ and $D$ share the \emph{same} vertex set, we write $F\subseteq D$ ($F$ is then called a \emph{spanning} subdigraph of $D$). If $S$ is a subet of the vertex set, the \emph{restriction of $D$ to $S$}, denoted $D|_S$, is the digraph with vertex set $S$ and with an arrow $(v,w){\blue\black \in S\times S}$ if and only if $(v,w)$ is an arrow in $D$. The \emph{complement} of a digraph has an arrow $e$ if and only if $D$ does not contain $e$, and is denoted $\overline{D}$. Note that $\overline{D|_S}=\overline{D}|_S$. A \emph{path} or \emph{cycle on $D$} is a subdigraph that is a (directed) path or cycle digraph. Paths and cycles are never allowed to use the same vertex twice (that is, no two arrows can point to or from the same vertex). A path may be an isolated vertex; a cycle may be a loop. If there are no cycles on $D$, then $D$ is said to be \emph{acyclic}. A path or cycle passing through \emph{every} vertex is said to be \emph{Hamiltonian}. A \emph{path cover} $E\stackrel{,}{\subseteq} D$ of a digraph $D$ is a digraph consisting of a disjoint union of paths on $D$ such that every vertex is included in exactly one path. A \emph{cycle cover} $F\stackrel{,}{\subseteq} D\black$ is defined similarly.

\begin{thm}\label{thm:peqc2}Suppose $D$ is an acyclic digraph. Then the number of Hamiltonian paths on $\overline{D}$ is equal to the number of cycle covers of $\overline{D}$.
\end{thm}
\begin{proof}This follows immediately from Equation \eqref{eq:pcgen} below, proved by Lass in \cite[Theorem 4.1]{MR1928099}, where it is attributed to Berge. It also follows immediately from Corollary \ref{thm:peqc1} below, which follows from Stanley and Stembridge's formula for the cycle symmetric function of the complement of a digraph \cite[Theorem 3.2]{MR1207737} and the fact that $\omega(\Pi_D)=\Pi_{\overline{D}}$ (see \cite[Corollary 2]{MR1375952} or Section \ref{sec:pathrec} below for a proof of the latter equality), or from Chow's reciprocity theorem \cite[Theorem 1]{MR1375952}. We will give here two direct combinatorial proofs.

\emph{First proof.} For a digraph $D$, we define the \emph{contraction} $D/e$ of $D$ by an arrow $e\stackrel{,}{=}(v,w)$ (not necessarily in $D$) to be the digraph obtained from $D$ by deleting all arrows pointing to $w$ or from $v$, and merging $v$ and $w$ into a single vertex. If $e$ is not an arrow of $D$, we denote by $D\cup e$ the digraph obtained by adding $e$ to $D$. Notice that if $D\cup e$ is acyclic, then so is $D/e$. Suppose $\alpha(D)$ is a positive integer defined for all acyclic digraphs $D$. Consider the following two conditions on $\alpha$:
\begin{enumerate}
\item If $e$ is not an arrow of $D$ and $D\cup e$ is acyclic, then $\alpha(D)=\alpha(D\cup e)+\alpha(D/e)$.
\item If $D$ is a complete acyclic digraph (an acyclic digraph with $\binom{|V|}{2}$ arrows), then $\alpha(D)=1$.
\end{enumerate}
Clearly, these conditions completely determine $\alpha(D)$. To see that both the number of cycle covers on $\overline{D}$ and the number Hamiltonian paths on $\overline{D}$ satisfy the first condition, notice that cycle covers or Hamiltonian paths on $\overline{D\cup e}$ are in bijection with those on $\overline{D}$ that do not use the arrow $e$, whereas those on $\overline{D/e}$ are in bijection with those that \emph{do} use $e$. The second condition is also easily verified: the only cycle cover of the complement of a complete acyclic digraph is the one consisting only of loops; the only Hamiltonian path on the complement of a complete acyclic  digraph is the one whose $k$'th vertex points to exactly $k-1$ other vertices.

It is possible to construct a bijection based on this inductive proof by inventing a structure to keep track of which arrow are contracted. A simpler bijection is the following.

\emph{Second proof.} The following algorithm is a refinement of Foata's \emph{first fundamental transformation} for permutations \cite[Chapter 10]{MR675953}. It has been used by Buhler and Graham in the case where $\overline{D}$ is the digraph corresponding to a poset. Fix a total order $R$ of the vertex set $V$ such that if $D$ contains an arrow from $v$ to $w$, then $v<w$ in $R$. The following bijection depends only on this order, not on $D$ (but the domain and range depend on $D$). We use the phrases ``greater than,'' ``less than,'' ``greatest,'' and ``least'' for this order, and the words ``before,'' ``after,'' ``first,'' and ``last'' in reference to a particular Hamiltonian path. We would like to construct a bijection between
\begin{itemize}
\item Hamiltonian paths on $\overline{D}$, and
\item ordered quadruples $(S,T,\rho,c)$, where $S\uplus T\stackrel{,}{=}V$ is an ordered bipartition of the vertex set with $T\neq\emptyset$, $\rho$ is a Hamiltonian path on $\overline{D|_S}$, $c$ is a Hamiltonian cycle on $\overline{D|_{T}}$, and (if $S\neq\emptyset$) the last vertex of $\rho$ is greater than every element of $T$.
\end{itemize}
If we have such a bijection satisfying the additional property that in the image of a path $E$ the set $T$ contains the last vertex of $E$, then this bijection can be iterated on the path $\rho$ until a cycle cover is obtained. For the inverse, we can proceed in reverse, starting with the cycle containing the largest vertex. See Figure \ref{bijfig} for an example of the full bijection; the image of each Hamiltonian path in the second row appears directly beneath it in the third row.

\begin{figure}
\psset{unit=3cm}
\begin{pspicture}(0.0,0.4)(2.7,1.9)
\psset{dotsize=0.06}
\psset{linewidth=0.02}
\psset{arrowscale=1.2}

\psdots(0.2,1.5)
\psdots(0.6,1.7)
\psdots(1.0,1.5)
\psdots(1.2,1.1)
\psdots(1.5,0.8)
\psdots(1.9,0.6)
\psdots(2.3,0.8)
\psdots(2.5,1.2)

\psline{->}(0.2,1.5)(0.6,1.7)
\psline{->}(0.6,1.7)(1.0,1.5)
\psline{->}(1.0,1.5)(1.2,1.1)
\psline{->}(1.5,0.8)(1.9,0.6)
\psline{->}(1.9,0.6)(2.3,0.8)
\psline{->}(2.3,0.8)(2.5,1.2)

\psline[linestyle=dashed]{->}(2.5,1.2)(1.5,0.8)

\uput[45](1.2,1.1){$v\stackrel{,}{>}z$}
\uput[45](2.5,1.2){$z$}
\uput[120](2.0,1.0){$e$}
\uput[-45](2.3,0.8){$<z$}
\uput[-90](1.9,0.6){$<z$}
\uput[-120](1.5,0.8){$<z$}

\put(0.7,1.3){$\rho$}
\put(2.0,0.8){$c$}
\end{pspicture}
\caption{The image of a Hamiltonian path under the bijection described in the second proof of Theorem \ref{thm:peqc2}.}\label{pathcyclefig}
\end{figure}
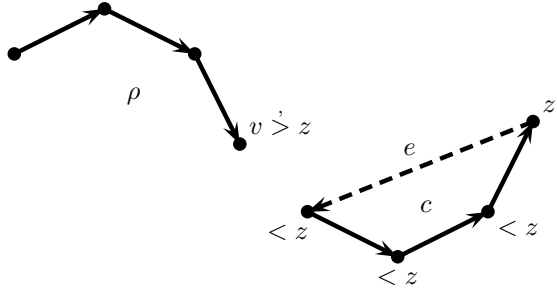 

So let $F$ be a Hamiltonian path on $\overline{D}$. Let $z$ be the last vertex of $F$. Let $v$ be the last vertex of $F$ such that $v>z$. Assuming such a $v$ exists, let $S$ be the set of vertices before and including $v$, and let $T$ be the set of vertices after $v$. If no such $v$ exists, let $S=\emptyset$, and let $T$ be the whole vertex set. Let $\rho=F|_S$. Let $c=F|_{T}\cup e$, where $e$ is the arrow connecting the ends of $F|_{T}$ to form a cycle (see Figure \ref{pathcyclefig}). The arrow $e=(z,u)$ cannot be in $D|_T$; otherwise, $z$ would be less than $u$.

For the inverse, suppose we have a quadruple $(S,T,\rho,c)$. Let $z$ be the greatest vertex in $T$. Assuming $S\neq\emptyset$, form a Hamiltonian path $F$ by starting with $\rho$, then moving along an arrow $f$ to the vertex just after $z$ in $c$, then following $c$ around to end at $z$. If the arrow $f\stackrel{,}{=}(v,w)$ were in $D$, we would have $v<w$, in contradiction to the assumption that $v$ is greater than every element of $T$. In the case where $S=\emptyset$, form $F$ by starting with the vertex just after $z$ in $c$, then following $c$ around to end at $z$.
\end{proof}

We use the notation of \cite{MR1354144} for partitions and symmetric functions. We will also need \emph{augmented monomial symmetric functions} $\tilde{m}_\lambda$, defined by $\tilde{m}_\lambda=(r_1!r_2!\cdots)m_\lambda$, where $\lambda=(1^{r_1}2^{r_2}\cdots)$. The number of parts of a partition $\lambda$ is denoted by $\ell(\lambda)$. All of our symmetric functions will be homogeneous, and we let $n=|\lambda|$ throughout. The \emph{type} $t(D)$ of a digraph $D$ is the partition whose parts are the sizes of the connected components of $D$. We abbreviate $\ell(t(D))$ to $\ell(D)$.

\begin{defn}Let $D$ be a digraph. Define the \emph{path symmetric function} $\Pi_D$ and \emph{cycle symmetric function} $Z_D$ by
\[
\Pi_D=\sum_{E\subseteq D}\tilde{m}_{t(E)},
\]
where the sum is over all path covers of $D$, and
\[
Z_D=\sum_{F\subseteq D}p_{t(F)},
\]
where the sum is over all cycle covers of $D$.
\end{defn}
The cycle symmetric function was introduced (for boards) by Stanley and Stembridge \cite{MR1207737}. (A cycle cover of a digraph is simply a placement of $n$ non-attacking rooks on the corresponding board; the analogues of paths and path covers, however, are inconvenient to define for boards.) Both the path and cycle symmetric functions are special cases of Chow's path-cycle symmetric function \cite{MR1375952}, which is a symmetric function generalization of Chung and Graham's cover polynomial \cite{MR1358990}.

\begin{cor}\label{thm:peqc1}Suppose $D$ is an acyclic digraph. Then $\Pi_{\overline{D}}=Z_{\overline{D}}$.
\end{cor}
\begin{proof}For a (not necessarily acyclic) digraph $H$, define $\pi_H$ to be the number of Hamiltonian paths on $H$, and define $z_H$ to be the number of Hamiltonian cycles on $H$. Let $V$ be the vertex set of $H$. Recall that the restriction of $H$ to $T\stackrel{,}{\subseteq} V$, denoted $H|_T$, is the digraph with vertex set $T$ and all arrows of $H$ with both ends in $T$. We have
\begin{equation}\label{eq:pth}
\Pi_{H}=\sum_{\sigma\vdash V}\tilde{m}_{t(\sigma)}\prod_{T\in\sigma}\pi_{H|_T}
\end{equation}
and
\begin{equation}\label{eq:zth}
Z_{H}=\sum_{\sigma\vdash V}p_{t(\sigma)}\prod_{T\in\sigma}z_{H|_T},
\end{equation}
where the sums are over all \emph{set} partitions of the vertex set $V$, and $t(\sigma)$ is the \emph{integer} partition whose parts are the sizes of the blocks of $\sigma$.

The following fundamental formula is proved in \cite{MR0429577}.
\[
p_{t(\sigma)}=\sum_{\gamma\geq\sigma}\tilde{m}_{t(\gamma)},
\]
where $\geq$ denotes the usual ordering of set partitions by refinement. Substituting into \eqref{eq:zth} with $H=\overline{D}$ and interchanging the order of summation yields
\begin{equation}\label{eq:zasso}
Z_{\overline{D}}=
\sum_{\gamma\vdash V}\tilde{m}_{t(\gamma)}
\sum_{\sigma\leq\gamma}\prod_{U\in\sigma}z_{\overline{D}|_U}\stackrel{,}{=}
\sum_{\gamma\vdash V}\tilde{m}_{t(\gamma)}\prod_{T\in\gamma}
\sum_{\sigma\vdash T}\prod_{U\in\sigma}z_{\overline{D}|_U}.
\end{equation}
But
\[
\sum_{\sigma\vdash T}\prod_{U\in\sigma}z_{\overline{D}|_U}=\pi_{\overline{D}|_T}
\]
by Theorem \ref{thm:peqc2}.
\end{proof}

\begin{note}\label{note:spm}
The computation \eqref{eq:zasso} is best understood in the following terms. A \emph{set partition map} is a function with domain the set of set partitions of finite subsets of a fixed set $V$. For a set partition map $f$, we use the notation $f_\sigma$ to denote the image of a set partition $\sigma\stackrel{,}{\vdash}S\stackrel{,}{\subseteq} V$. When the range is a ring, \emph{composition} of set partition maps is defined by
\[
(f\circ g)_\sigma=\sum_{\pi\geq\sigma}f_\pi\prod_{T\in\pi}g_{\sigma|_T},
\]
where $\geq$ denotes the usual ordering of set partitions by refinement, and $\sigma|_T$ is the set partition of $T$ consisting of the blocks of $\sigma$ contained in $T$. The computation \eqref{eq:zasso} is then essentially the computation that shows composition to be associative.
\end{note}

If we take $D$ to be an empty graph, Theorem \ref{thm:peqc2} follows from the fact that a permutation can be written as a word or as a collection of cycles. This fact can be used to enumerate rooted trees \cite[Example 12]{MR633783}. Theorem \ref{thm:peqc2} gives the following generalization. A \emph{directed tree} with root $r$ is a digraph such that there is a unique (directed) path from every vertex to $r$. A \emph{directed spanning tree} on a digraph $D$ is a spanning subdigraph of $D$ that is a directed tree.
\begin{cor}\label{cor:rtree}
Let $D$ be an acyclic digraph. For a vertex $v\stackrel{,}{\in} V$, let $d(v)$ be the number of arrows in $\overline{D}$ pointing from $v$. Then the number of directed trees on $\overline{D}$ is
\[
\frac{1}{|V|}\prod_{v\in V}d(v).
\]
\end{cor}
\begin{proof}
We will describe a bijection between
\begin{itemize}
\item
ordered pairs $(v, T)$, where $v\in V$ and $T$ is a directed tree on $\overline{D}$, and
\item
digraphs $B\stackrel{,}{\subseteq} \overline{D}$ with exactly one arrow pointing from each vertex.
\end{itemize}
Since there are $d(v)$ choices of an arrow pointing from $v$, the result would follow from such a bijection.

So suppose we have a pair $(v,T)$. Let $F$ be the unique path from $v$ to the root $r$ of $T$. Replace $F$ with the corresponding cycle cover given by the bijection of Theorem \ref{thm:peqc2} (using the ordering of the vertices of $F$ compatible with a suitable fixed ordering of $V$). The resulting digraph has exactly one arrow pointing from each vertex. For the inverse, reverse this process, applying the inverse of the bijection of Theorem \ref{thm:peqc2} to the collection of cycles in $B$.
\end{proof}


\begin{note}
Let $H$ be any digraph. The following identity, which was described in the introduction, generalizes Theorem \ref{thm:peqc2} to the case of an arbitrary (not necessarily acyclic) digraph. It is proved by Lass in \cite[Theorem 4.1]{MR1928099}, where it is attributed to Berge.
\begin{equation}\label{eq:pcgen}
\pi_H=\sum_{\sigma\vdash V}\prod_{T\in\sigma}\left(z_{H|_T}-\left(-1\right)^{|T|}z_{\overline{H}|_T}\right).
\end{equation}
In the case where $H$ is the complement of an acyclic graph, there are no cycle covers of $\overline{D}$. Hence we are left with Theorem \ref{thm:peqc2}. Also, in the case where $H$ itself is acyclic, \eqref{eq:pcgen} implies that the determinant of the adjacency matrix of the complement of an acyclic digraph $H$ is the number of Hamiltonian paths on $H$ (which is $0$ or $1$). This result has been used by Stanley \cite{MR1383514}.
\end{note}

\section{$e$-coefficients of incomparability graphs}

The main result of this section is based on the following lemma, which gives a combinatorial interpretation of the coefficients in the expansion of power sum symmetric functions in terms of elementary symmetric functions.
\begin{lem}\label{lem:eco}
Let $\tau_\lambda$ be a digraph consisting of $\ell(\lambda)$ disjoint directed cycles whose lengths are equal to the parts of $\lambda$ (so $t(\tau(\lambda))=\lambda$). Then
\[
p_\lambda=(-1)^n\sum_{\substack{E\subseteq \tau_\lambda\\\mathrm{acyclic}}}(-1)^{\ell(E)}e_{t(E)}\stackrel{,}{=}(-1)^n\sum_{\mu}(-1)^{\ell(\mu)}c_\lambda^\mu e_\mu,
\]
where $c_\lambda^\mu$ is the number of acyclic spanning subdigraphs of $\tau_\lambda$ of type $\mu$.
\end{lem}
\begin{proof}It will suffice to prove the case where $\lambda$ has only one part. We have
\[
p_k=
\sum_{\substack{i_1,i_2,\ldots,i_k\\i_1=i_2=\cdots=i_k}}x_{i_1} x_{i_2}\cdots x_{i_k}\stackrel{,}{=}
\sum_{\substack{i_1,i_2,\ldots,i_k\\i_1\leq i_2\leq\cdots\leq i_k\leq i_1}}x_{i_1} x_{i_2}\cdots x_{i_k}.
\]
Using the principle
\[
\sum_\leq=\sum-\sum_>
\]
on all occurrences of $\leq$ beneath this sum, we obtain
\[
p_k=(-1)^k\sum_{\substack{E\subseteq \tau_k\\\mathrm{acyclic}}}(-1)^{\ell(E)}e_{t(E)}.
\]
For example,
\[
p_2=\sum_{\substack{i_1,i_2\\i_1=i_2}}x_{i_1} x_{i_2}\stackrel{,}{=}
\sum_{\substack{i_1,i_2\\i_1\leq i_2\leq i_1}}x_{i_1} x_{i_2},
\]
which expands into four sums:
\[
\sum_{i_1,i_2}x_{i_1} x_{i_2}-
\sum_{\substack{i_1,i_2\\i_1>i_2}}x_{i_1} x_{i_2}-
\sum_{\substack{i_1,i_2\\i_2>i_1}}x_{i_1} x_{i_2}+
\sum_{\substack{i_1,i_2\\i_1>i_2>i_1}}x_{i_1} x_{i_2}.
\]
The first three sums correspond to the three acyclic spanning sudigraphs of $\tau_2$. The fourth sum is $0$. Since the first sum is equal to $e_{11}$ and the second two are equal to $e_2$, we have
\[
p_2=e_{11}-2e_2.
\]
\end{proof}

To state our next result we will need the following more general notion of contraction.
\begin{defn}
Let $D$ be a digraph and $E\stackrel{,}{\subseteq} D$ a path cover. We define $D/E$ to be the digraph with vertices the components of $E$ (which are paths) and with an arrow from $v$ to $w$ if and only if $D$ has an arrow from the end of the path $v$ to the beginning of the path $w$.
\end{defn}

The determinant of a square matrix $A\stackrel{,}{=}(a_{i,j})_{1\leq i,j\leq n}$ is defined by
\[
\mathrm{det}(A)=\sum_{\omega}\mathrm{sign}(\omega)\prod_{i=1}^n a_{i,\omega(i)},
\]
where the sum is over all permutations of $\{1,2,\ldots,n\}$ and $\mathrm{sign}(\omega)$ is the sign of the permutation $\omega$.

\begin{thm}\label{thm:zexp}Let $D$ be a digraph.
\begin{enumerate}
\item
The coefficient of $(-1)^{n-\ell(\lambda)}e_\lambda$ in $Z_D$ is the number of ordered pairs $(E,F)$, where $E\subseteq D$ is a path cover of $D$ of type $\lambda$ and $F\subseteq D/E$ is a cycle cover of $D/E$.
\item
The coefficient of $h_\lambda$ in $Z_D$ is
\[
\sum_{\substack{E\subseteq D\\t(E)=\lambda}}\mathrm{det}(D/E),
\]
where the sum is over all path covers of $D$ of type $\lambda$ and $\mathrm{det}(D/E)$ is the determinant of the adjacency matrix of $D/E$.
\end{enumerate}
\end{thm}
Before proving this theorem, let us consider two examples of Part 2.
\begin{ex}For a one-part partition $\lambda=(n)$, a path cover of $D$ of type $\lambda$ is a simply a Hamiltonian path on $D$. In this case, the digraph $D/E$ contains a single vertex $v$, and the loop $(v,v)$ if and only if the arrow connecting the last vertex of $E$ to the first vertex of $E$ is in $D$. Hence the coefficient of $h_n$ in $Z_D$ is the number of Hamiltonian cycles on $D$ with one edge missing. We call these \emph{broken} Hamiltonian cycles on $D$ (see Section \ref{sec:ss}).
\end{ex}
\begin{ex}For a two-part partition $\lambda=(a,b)$, a path cover of $D$ of type $\lambda$ is a pair $E=E_1\cup E_2$ of disjoint paths on $D$ such that $E_1$ has $a$ vertices and $E_2$ has $b$ vertices. We call such a cover an \emph{$ab$-path cover} of $D$. In this case, the digraph $D/E$ contains two vertices. There are two possible cycle covers of a $2$-vertex digraph: two loops ($\alpha$) or a $2$-vertex cycle ($\beta$). The determinant of the $2\times2$ adjacency matrix of $D/E$ is $1$ if $D/E$ contains only $\alpha$, it is $-1$ if $D/E$ contains only $\beta$, and it is $0$ otherwise. We say that an $ab$-path cover is \emph{like} $\alpha$ if each path can be connected end-to-beginning to form a pair of cycles on $D$. We say that it is like $\beta$ if the end of $E_1$ can be connected to the beginning of $E_2$ and the end of $E_2$ to the beginning of $E_1$ to form a single Hamiltonian cycle on $D$. Then the coefficient of $h_{ab}$ in $Z_D$ is the difference of
\begin{itemize}
\item the number of $ab$-path covers like $\alpha$ but not like $\beta$, and
\item the number of $ab$-path covers like $\beta$ but not like $\alpha$.
\end{itemize}
\end{ex}

\begin{proof}[Proof of Theorem \ref{thm:zexp}]From the definition of $Z_D$ and Lemma \ref{lem:eco}, the coefficient of $e_\lambda$ is
\[
(-1)^{n-\ell(\lambda)}\sum_{F\subseteq D}
\sum_{\substack{E\subseteq F\\\mathrm{acyclic}\\t(E)=\lambda}}1,
\]
where the outer sum is over all cycle covers of $D$. An acyclic spanning subgraph of $F$ is clearly a path cover, so after interchanging the order of summation this becomes
\[
(-1)^{n-\ell(\lambda)}
\sum_{\substack{E\subseteq D\\t(E)=\lambda}}
\sum_{\substack{E\subseteq F\subseteq D}}1,
\]
where the outer sum is over all path covers of $D$ of type $\lambda$ and the inner sum is over all cycle covers of $D$ containing $E$ as a spanning subgraph. Such cycle covers are in bijection with cycle covers of $D/E$. This proves the first part.

From the definition of $Z_D$ and Lemma \ref{lem:eco}, the coefficient of $e_\lambda$ in $\omega(Z_D)$ is
\[
(-1)^{\ell(\lambda)}\sum_{F\subseteq D}(-1)^{\ell(F)}
\sum_{\substack{E\subseteq F\\\mathrm{acyclic}\\t(E)=\lambda}}1,
\]
where the outer sum is over all cycle covers of $D$. As before, this becomes
\[
(-1)^{\ell(\lambda)}
\sum_{\substack{E\subseteq D\\t(E)=\lambda}}
\sum_{\substack{E\subseteq F\subseteq D}}(-1)^{\ell(F)},
\]
where the outer sum is over all path covers of $D$ of type $\lambda$ and the inner sum is over all cycle covers of $D$ containing $E$ as a spanning subgraph. Using the standard correspondence between cycle covers and permutations (a permutation can be written as a disjoint union of cycles), we see that the cycle covers of $D/E$ correspond to the permutations that contribute nonzero summands to the determinant of the adjacency matrix of $D/E$. The sign of a permutation is $(-1)^{r-k}$, where $r$ is the number of vertices and $k$ is the number of cycles; in the present case, $r=\ell(E)\stackrel{,}{=}\ell(\lambda)$ and $k=\ell(F)$. Hence
\[
(-1)^{\ell(\lambda)}
\sum_{\substack{E\subseteq F\subseteq D}}(-1)^{\ell(F)}=
(-1)^{\ell(\lambda)}\sum_{\substack{F\subseteq D/E}}(-1)^{\ell(F)}\stackrel{,}{=}\mathrm{det}(D/E),
\]
where the second sum is over all cycle covers of $D/E$.
\end{proof}

If $P$ is a poset with underlying set $V$ (we call elements of $V$ vertices), its \emph{incomparability graph} $\mathrm{inc}(P)$ is the graph with an edge between $v$ and $w$ if and only if $v$ and $w$ are incomparable in $P$. A poset also induces an acyclic digraph on $V$ with an arrow from $v$ to $w$ if and only if $v<w$ in $P$. Both the poset and this digraph will be denoted by $P$. A \emph{weakly decreasing} path or cycle on $P$ is a (directed) path or cycle such that consecutive elements are either incomparable or decreasing (so it is a path or cycle on $\overline{P}$). Likewise, a weakly decreasing path or cycle cover of $P$ is a path or cycle cover of $\overline{P}$. A loop is considered to be a weakly decreasing cycle. We will often say, ``the arrow from $v$ to $w$ is weakly decreasing,'' to mean that $v$ is not less than $w$, even if there is no digraph under consideration that actually contains that arrow.

Let $G$ be a graph with vertex set $V$. A subset $S\subseteq V$ is \emph{stable} if $G$ contains no edges between the the vertices of $S$. A \emph{stable partition} of $V$ is a set partition in which every block is stable. The \emph{chromatic symmetric function} \cite{MR96b:05174} of $G$ is given by
\[
X_G=\sum_{\substack{\sigma\vdash V\\\mathrm{stable}}}\tilde{m}_{t(\sigma)},
\]
where the sum is over all stable partitions of $G$. The following trivial proposition has been used by Chow \cite[Proposition 2]{MR1375952}.
\begin{prop}\label{prop:incpath}Let $P$ be a poset. Then
$\Pi_P=X_{\mathrm{inc}(P)}$.
\end{prop}
\begin{proof}
A path cover of $P$ is a partition of the vertex set into chains (fully ordered subsets). But a chain in $P$ is precisely a stable subset in $\mathrm{inc}(P)$.
\end{proof}

\begin{cor}\label{cor:eco}The coefficient of $e_\lambda$ in $X_{\mathrm{inc}(P)}$ is
\[
\sum_{\substack{E\subseteq\overline{P}\\t(E)=\lambda}}\mathrm{det}(\overline{P}/E),
\]
where the sum is over all weakly decreasing path covers of $P$ of type $\lambda$ and $\mathrm{det}(\overline{P}/E)$ is the determinant of the adjacency matrix of $\overline{P}/E$.
\end{cor}
\begin{proof}
By \cite[Corollary 2]{MR1375952}, $\omega(\Pi_D)=\Pi_{\overline{D}}$ (see also Section \ref{sec:pathrec} below). Hence
\[
X_{\mathrm{inc}(P)}=\Pi_P\stackrel{,}{=}\omega(\Pi_{\overline{P}})\stackrel{,}{=}\omega(Z_{\overline{P}}).
\]
\end{proof}

\subsection{The shatter bijection}

Let $P$ be a poset. In this section we will establish a bijection between weakly decreasing Hamiltonian paths on $P$ and acyclic orientations of $\mathrm{inc}(P)$. The map we will describe and call the \emph{shatter bijection} is familiar in trace theory (see~\cite{MR1075995}) as giving a correspondence between the \emph{lexicographic normal form} and the \emph{dependence graph} of a trace. Our contribution amounts to saying that, for the incomparability graph of a poset,  the lexicographic normal form derived from any extension of the poset to a total order is unique.

Suppose $F$ is a weakly decreasing Hamiltonian path. We construct an acyclic orientation $r(F)$ of ${\mathrm{inc}(P)}$ as follows. Let $v_1$ be the first vertex of $F$. Direct all edges of ${\mathrm{inc}(P)}$ adjacent to $v_1$ toward $v_1$. Now let $v_2$ be the second vertex of $F$, and again direct all not yet directed edges of ${\mathrm{inc}(P)}$ adjacent to $v_2$ toward $v_2$. Continue in this manner until all edges have been directed (see Figure \ref{shatterfig}, where the acyclic orientation $A$ is the image under $r$ of the weakly decreasing Hamiltonian path labeled $s(A)$).

Note that we could also define $r(F)$ by starting with \emph{last} vertex of $F$, and at each step directing \emph{all} edges adjacent to $v_i$ toward $v_i$, not just the not yet directed edges.

Now suppose $A$ is an acyclic orientation of ${\mathrm{inc}(P)}$. We construct a weakly decreasing Hamiltonian path $s(A)$ as follows. Let $S_1$ be the set of sinks in $A$ ($S_1$ is nonempty because $A$ is acyclic). Since there cannot be arrows between sinks, $S_1$ is totally ordered by $P$. Let $v_1$ be the greatest element of $S_1$. Now let $S_2$ be the set of sinks in $A|_{V\setminus v_1}$, and let $v_2$ be the greatest element of $S_2$. Continuing this process, with $S_i$ being the set of sinks in $A|_{V\setminus\{v_1,v_2,\ldots,v_{i-1}\}}$, until the vertex set is exhausted, we obtain a Hamiltonian path with vertex sequence $v_1, v_2, \ldots, v_n$. We call $s(A)$ the \emph{shatter-path} of $A$ (see Figure \ref{shatterfig}).

To show that $s(A)$ is weakly decreasing, it will suffice to show that in the above construction, $v_1$ is not less than $v_2$. So suppose $v_1<v_2$. Then $v_2$ must not be a sink in $A$; otherwise it would have been chosen as $v_1$. Since $v_2$ \emph{is} a sink in $A|_{V\setminus v_1}$, the deletion of $v_1$ must have destroyed an arrow pointing from $v_2$ to $v_1$. Hence there must be an edge in ${\mathrm{inc}(P)}$ between $v_1$ and $v_2$, a contradiction to the assumption that $v_1<v_2$.

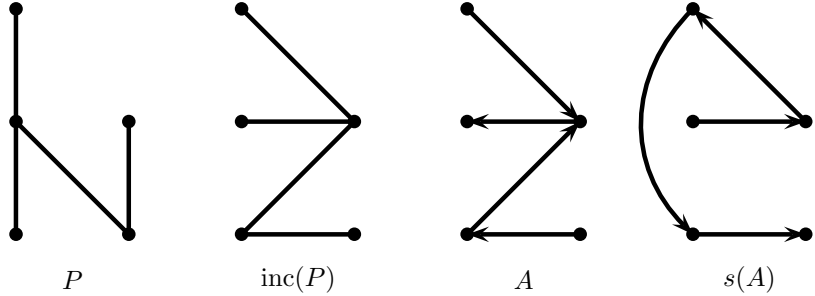
\begin{figure}
\psset{unit=3cm}
\begin{pspicture}(-0.2,-0.3)(3,1.2)
\psset{dotsize=0.06}
\psset{linewidth=0.02}
\psset{arrowscale=1.2}

\psdots(0.0,0.0)
\psdots(0.0,0.5)
\psdots(0.0,1.0)
\psdots(0.5,0.0)
\psdots(0.5,0.5)

\psline{-}(0.0,0.0)(0.0,0.5)
\psline{-}(0.0,0.5)(0.0,1.0)
\psline{-}(0.5,0.0)(0.5,0.5)
\psline{-}(0.5,0.0)(0.0,0.5)

\rput(0.25,-0.20){$P$}

\psset{origin={-1.0,0}}
\psdots(0.0,0.0)
\psdots(0.0,0.5)
\psdots(0.0,1.0)
\psdots(0.5,0.0)
\psdots(0.5,0.5)

\psline{-}(0.5,0.0)(0.0,0.0)
\psline{-}(0.0,0.0)(0.5,0.5)
\psline{-}(0.5,0.5)(0.0,0.5)
\psline{-}(0.0,1.0)(0.5,0.5)

\rput(1.25,-0.20){${\mathrm{inc}(P)}$}

\psset{origin={-2.0,0}}
\psdots(0.0,0.0)
\psdots(0.0,0.5)
\psdots(0.0,1.0)
\psdots(0.5,0.0)
\psdots(0.5,0.5)

\psline{->}(0.5,0.0)(0.0,0.0)
\psline{->}(0.0,0.0)(0.5,0.5)
\psline{->}(0.5,0.5)(0.0,0.5)
\psline{->}(0.0,1.0)(0.5,0.5)

\rput(2.25,-0.20){$A$}

\psset{origin={-3.0,0}}
\psdots(0.0,0.0)
\psdots(0.0,0.5)
\psdots(0.0,1.0)
\psdots(0.5,0.0)
\psdots(0.5,0.5)

\psline{->}(0.0,0.5)(0.5,0.5)
\psline{->}(0.5,0.5)(0.0,1.0)
\psbezier{->}(0.0,1.0)(-0.3,0.7)(-0.3,0.3)(0.0,0.0)
\psline{->}(0.0,0.0)(0.5,0.0)

\rput(3.25,-0.20){$s(A)$}

\end{pspicture}
\caption{The Hasse diagram of a poset $P$, its incomparability graph ${\mathrm{inc}(P)}$, an acyclic orientation of ${\mathrm{inc}(P)}$, and its weakly decreasing shatter-path.}\label{shatterfig}
\end{figure} 

\begin{thm}\label{patho}
The map $r$ is a bijection between weakly decreasing Hamiltonian paths on $P$ and acyclic orientations of ${\mathrm{inc}(P)}$, with inverse $s$.
\end{thm}
\begin{proof}
Let $A$ be an acyclic orientation of ${\mathrm{inc}(P)}$, with greatest sink $v$. Define $\psi(A)=(v,A|_{V\setminus v})$. It is straightforward to show that $\psi$ is a bijection between
\begin{itemize}
\item acyclic orientations of ${\mathrm{inc}(P)}$, and
\item pairs $(v, B)$, where $v\stackrel{,}{\in} V$ is a vertex, $B$ is an acyclic orientation of ${\mathrm{inc}(P)}|_{V\setminus v}$, and the arrow from $v$ to the greatest sink in $B$ is weakly decreasing.
\end{itemize}
Let $n=|V|$. Notice that the map $s$ is obtained by iterating $\psi$ on the remaining orientation: the $n$'th iteration of $\psi$ is a bijection between
\begin{itemize}
\item acyclic orientations of ${\mathrm{inc}(P)}$, and
\item sequences $(v_1, v_2, \ldots, v_n)$ of distinct, weakly decreasing vertices.
\end{itemize}
Of course, such sequences are simply weakly decreasing Hamiltonian paths. Since $\psi^{-1}(B,v)$ is obtained from $B$ by adding the vertex $v$ and directing all adjacent edges in ${\mathrm{inc}(P)}$ toward $v$, we see that $r$ is obtained by iterating $\psi^{-1}$, starting with $v$ as the last vertex of the weakly decreasing path and proceeding backward along the path toward the first vertex.
\end{proof}

\subsection{The second-sink bijection}\label{sec:ss}

A \emph{broken cycle} is a cycle with one arrow removed. So a weakly decreasing path is a broken weakly decreasing cycle if and only if the cycle obtained by adding an arrow from its end to its beginning is weakly decreasing. Corollary \ref{cor:eco} implies that the coefficient of $e_n$ in $X_{\mathrm{inc}(P)}$ is the number of broken weakly decreasing Hamiltonian cycles on $P$. For a poset $P$, let $c_P$ be the number of weakly decreasing Hamiltonian cycles on $P$. Let $v\stackrel{,}{\in} V$ be a fixed vertex. It is clear that $c_P$ is also the number of broken weakly decreasing Hamiltonian cycles beginning at $v$.

Let $a_P$ be the coefficient of $p_n$ in $X_{\mathrm{inc}(P)}$. It follows from Lemma \ref{lem:eco} that $c_P=(-1)^{n-1}a_P$. It is known that $a_P$ is the number of acyclic orientations of $\mathrm{inc}(P)$ with a unique sink at a fixed vertex (see the next section for a proof). Hence there should be a bijection between broken weakly decreasing Hamiltonian cycles beginning at $v$ and acyclic orientations of $\mathrm{inc}(P)$ with a unique sink at $v$. We call an acyclic orientation of $\mathrm{inc}(P)$ \emph{circular} if its shatter-path is a broken weakly decreasing cycle, i.e., if its smallest source is not less than its greatest sink. By Theorem \ref{patho}, such a bijection would follow from a bijection between circular acyclic orientations of $\mathrm{inc}(P)$ with greatest sink $v$ and acyclic orientations of $\mathrm{inc}(P)$ with a unique sink at $v$. Such a bijection is outlined in this section. Actually, there is a much stronger relationship between $e$-coefficients and acyclic orientations \cite[Theorems 3.3 and 3.4]{MR96b:05174} which we do not consider here; the combinatorics of this relationship deserves further investigation.

Let $A$ be an acyclic orientation. If $v$ is a sink of $A$, we can form a new acyclic orientation by reversing every arrow adjacent to $v$. This is described as \emph{flipping} the sink $v$. For more on flipping sinks, see \cite{math.CO/0209005}.

Assume $\mathrm{inc}(P)$ is connected. Let $A$ be a circular acyclic orientation of $\mathrm{inc}(P)$. The set of sinks in $A$ is totally ordered by $P$. So if $A$ has more than one sink, we can flip the second-largest sink. Let $t(A)$ be the orientation resulting from repeated flipping of the second sink until an orientation having only one sink is obtained (see Figure \ref{secondsinkfig}). Because the largest sink is never flipped, this process must terminate.

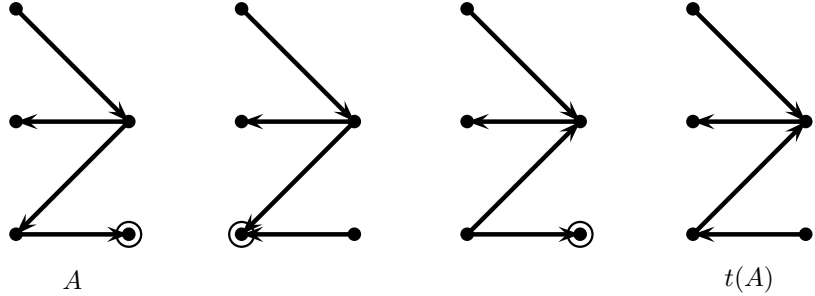
\begin{figure}
\psset{unit=3cm}
\begin{pspicture}(-0.2,-0.3)(3,1.2)
\psset{dotsize=0.06}
\psset{linewidth=0.02}
\psset{arrowscale=1.2}

\psset{origin={0,0}}
\psdots(0.0,0.0)
\psdots(0.0,0.5)
\psdots(0.0,1.0)
\psdots(0.5,0.0)
\psdots(0.5,0.5)
\pscircle[linewidth=0.01](0.5,0.0){0.06}

\psline{<-}(0.5,0.0)(0.0,0.0)
\psline{<-}(0.0,0.0)(0.5,0.5)
\psline{->}(0.5,0.5)(0.0,0.5)
\psline{->}(0.0,1.0)(0.5,0.5)

\rput(0.25,-0.20){$A$}

\psset{origin={-1.0,0}}
\psdots(0.0,0.0)
\psdots(0.0,0.5)
\psdots(0.0,1.0)
\psdots(0.5,0.0)
\psdots(0.5,0.5)
\pscircle[linewidth=0.01](0.0,0.0){0.06}

\psline{->}(0.5,0.0)(0.0,0.0)
\psline{<-}(0.0,0.0)(0.5,0.5)
\psline{->}(0.5,0.5)(0.0,0.5)
\psline{->}(0.0,1.0)(0.5,0.5)

\psset{origin={-2.0,0}}
\psdots(0.0,0.0)
\psdots(0.0,0.5)
\psdots(0.0,1.0)
\psdots(0.5,0.0)
\psdots(0.5,0.5)
\pscircle[linewidth=0.01](0.5,0.0){0.06}

\psline{<-}(0.5,0.0)(0.0,0.0)
\psline{->}(0.0,0.0)(0.5,0.5)
\psline{->}(0.5,0.5)(0.0,0.5)
\psline{->}(0.0,1.0)(0.5,0.5)

\psset{origin={-3.0,0}}
\psdots(0.0,0.0)
\psdots(0.0,0.5)
\psdots(0.0,1.0)
\psdots(0.5,0.0)
\psdots(0.5,0.5)

\psline{->}(0.5,0.0)(0.0,0.0)
\psline{->}(0.0,0.0)(0.5,0.5)
\psline{->}(0.5,0.5)(0.0,0.5)
\psline{->}(0.0,1.0)(0.5,0.5)

\rput(3.25,-0.20){$t(A)$}

\end{pspicture}
\caption{A sequence of acyclic orientations of $\mathrm{inc}(P)$, starting with a circular orientation $A$ and ending with a one-sink orientation $t(A)$. The second-largest sink is circled in each orientation.}\label{secondsinkfig}
\end{figure} 

On the other hand, let $B$ be an acyclic orientation with just one sink. If $B$ is not circular, flip the smallest source. Repeat this process until a circular orientation is obtained. Since $\mathrm{inc}(P)$ is connected, this process must terminate (we omit the proof). Let $u(B)$ be the resulting orientation.

\begin{prop}\label{prop:sec}
Let $P$ be a poset such that $\mathrm{inc}(P)$ is connected, and let $v\stackrel{,}{\in} V$ be a vertex. Then the function $t$ is a bijection between
\begin{itemize}
\item
circular acyclic orientations of $\mathrm{inc}(P)$ with greatest sink $v$, and
\item
acyclic orientations of $\mathrm{inc}(P)$ with a unique sink at $v$,
\end{itemize}
with inverse $u$.
\end{prop}
\begin{proof}
Let $A$ be a circular acyclic orientation of $\mathrm{inc}(P)$. The key observation is that each time the second-largest sink is flipped, it becomes the smallest source. To see this, let $a$ be the largest sink, $v$ the smallest source, and $w$ the second-largest sink. Let $u$ be any source (other than $w$) in the orientation after $w$ is flipped (if $w$ is the only source, we have nothing to prove). Then $u$ is also a source of $A$. If $w>u$, then $w>v$ because $u>v$. But $w<a$, so this would imply $v<a$, in contradiction to being circular. Now let $w^\prime$ be the new second-largest sink. As before, $w^\prime$ must be less than any source of $A$ except possibly $w$. If $w^\prime$ were greater than $w$, then $w^\prime$ would be the second-largest sink of $A$, not $w$. If $w^\prime$ is incomparable to $w$, then $w$ is no longer a source after $w^\prime$ is flipped. And so on.

Let $B$ be an acyclic orientation with just one sink. The key observation here is that each time the smallest source is flipped, it becomes the second-largest sink. To see this, let $a$ be the largest sink, and $v$ the smallest source, and $v^\prime$ the smallest source after $v$ is flipped. If $v^\prime$ were less than $v'$, then $v^\prime$ would be the smallest source of $A$, not $v$. And so on.
\end{proof}

This proposition illustrates why the number of acyclic orientations of an incomparability graph with a unique sink at a fixed vertex does not depend on the choice of a fixed vertex: a broken cycle can be cycled to a broken cycle beginning at any vertex.

Note that using Theorem \ref{patho} we could have described this as a bijection between broken weakly decreasing Hamiltonian cycles beginning at $v$ and weakly decreasing Hamiltonian paths beginning at $v$ with no vertex less than all previous vertices.

\section{Stable link sequences and acyclic orientations}
 
For a graph $G$, let $a_G$ be the coefficient of $p_n$ in its chromatic symmetric function $X_G$. So the $a_P$ of the previous section is now $a_{\mathrm{inc}(P)}$. It is known that $(-1)^{n-1}a_G$ is the number of acyclic orientations of $G$ with a unique sink at a fixed vertex. This result was first proved in \cite{MR712251} and was given more directly combinatorial proofs in \cite{MR1778205}. In this section we will give another proof, based on a stronger result.

Let $S$ be a subset of $V$. We say that a graph (undirected) with vertex set $V$ is \emph{anchored by $S$} if for every $v\stackrel{,}{\in} V$ there is a path between $v$ and some element of $S$. We define a function $j_S$ from graphs on $V$ anchored by $S$ to ordered set partitions of $V$ as follows. Let $G$ be a graph on $V$ anchored by $S$. Begin by setting $\sigma_1=S$. If $\sigma_1$ is not equal to $V$, let $\sigma_2$ be the subset of $V\setminus\sigma_1$ of vertices adjacent to some element of $\sigma_1$. If $\sigma_1\cup\sigma_2$ is not equal to $V$, let $\sigma_3$ be the subset of $V\setminus\sigma_1\setminus\sigma_2$ of vertices adjacent to some element of $\sigma_2$. Continuing in this manner until every vertex is contained in some $\sigma_i$, we obtain an ordered set partition $\sigma\stackrel{,}{=}(\sigma_1, \sigma_2, \ldots, \sigma_r)$ of $G$. Set $j_S(G)=\sigma$.

The following proposition states that choosing a graph $G$ such that $j_{\sigma_1}(G)=\sigma$ is equivalent to choosing, for each $i$, a set (possibly empty) of edges with both ends in $\sigma_i$, and for each $v\stackrel{,}{\in} V\setminus\sigma_1$, a nonempty set of edges linking $v$ to the block of $\sigma$ preceding the one containing $v$. The simple proof is omitted.

\begin{prop}Let $G$ be a graph on $V$ anchored by $\sigma_1$ and $\sigma\stackrel{,}{=}(\sigma_1, \sigma_2, \ldots, \sigma_r)$ an ordered set partition of $V$. Then the following are equivalent.
\begin{itemize}
\item $j_{\sigma_1}(G)=\sigma$.
\item The edge set of $G$ is the union of the elements of two families of sets, $(A_i)_{i=1}^r$ and $(B_v)_{v\in V\setminus \sigma_1}$, where
\begin{itemize}
\item $A_i$ is a set of edges with both ends in $\sigma_i$, and
\item $B_v$ is a nonempty set of edges linking $v$ to $\sigma_{s-1}$, where $\sigma_s$ is the block of $\sigma$ containing $v$.
\end{itemize}
\end{itemize}
\end{prop} 

Let $\eta_G(t)$ be  the polynomial whose coefficient of $t^k$ is the number of connected spanning subgraphs of $G$ with $k$ edges (recall that a spanning subgraph is one with the same vertex set). The subgraph expansion for the chromatic symmetric function \cite[Theorem 2.5]{MR96b:05174} implies that $a_G=\eta_G(-1)$.

Let $v_0\stackrel{,}{\in} V$ be a fixed vertex. Observe that for a graph $G$ to be anchored by $S\stackrel{,}{=}\{v_0\}$ is equivalent to $G$ being connected. If $\sigma$ is an ordered set partition of $V$ and $v\in V\setminus\sigma_1$, let $h_G(\sigma,v)$ be the number of edges in $G$ between $v$ and the block of $\sigma$ preceding the block containing $v$. The previous proposition yields
\[
\eta_G(t)=
\sum_{\substack{\sigma\\\sigma_1=\{v\}}}\quad
\prod_{i=1}^{\ell(\sigma)}(1+t)^{|E(G|_{\sigma_i})|}
\prod_{v\in V\setminus\{v\}}[(1+t)^{h_G(\sigma,v)}-1],
\]
where the sum is over all ordered set partitions $\sigma$ such that $\sigma_1=\{v\}$, and $\ell(\sigma)$ is the number of blocks. In particular, $(-1)^{|V|-1}\eta_G(-1)$ is the number of ordered set partitions satisfying $\sigma_1=\{v\}$ and the condition of the following definition. Recall that a stable subset $S$ of $V$ is one such that $G|_S$ contains no edges.

\begin{defn}Let $G$ be a graph on $V$. We call an ordered set partition $\sigma$ of $V$ a \emph{stable link sequence} of $G$ if each $\sigma_i$ is a stable set of $G$ and if for each $v\stackrel{,}{\in} V\setminus\sigma_1$, there is an edge in $G$ between $v$ and some element of the block of $\sigma$ preceding the block containing $v$.
\end{defn}

Let $G$ be a graph on $V$ and $A$ an acyclic orientation of $G$. An ordered set partition $f(A)$ of $V$, called the \emph{sink sequence} of $A$, can be constructed as follows. Let $\sigma_1$ be the set of sinks in $A$. When these sinks are removed (along with all incident edges), we are left with a new acyclic digraph $A_2\stackrel{,}{=}A|_{V\setminus\sigma_1}$. Let $\sigma_2$ be the set of sinks in $A_2$. Since an acyclic digraph must have at least one sink, this process may be continued until every vertex is contained in some block of $\sigma$.

Notice that if $\sigma_1$ is the set of sinks in $A$ and $\sigma_2$ is the set of sinks in $A_2\stackrel{,}{=}A|_{V\setminus\sigma_1}$, then for each $v\stackrel{,}{\in}\sigma_2$, $A$ has an edge between $v$ and some element of $\sigma_1$. If not, $v$ would be a sink of $A$ rather than $A_2$. Since the set of sinks in an acyclic orientation is necessarily stable, it follows that $f(A)$ is a stable link sequence.

\begin{thm}\label{thm:sls}The map $f$ from acyclic orientations of $G$ to stable link sequences of $G$ is a bijection.
\end{thm}
\begin{proof}We will define a map $\psi$ from acyclic orientations $A$ of $G$ to pairs $(S,A^\prime)$, where
\begin{itemize}
\item$S$ is a stable subset of $V$,
\item$A^\prime$ is an acyclic orientation of $G|_{V\setminus S}$, and
\item there is an edge in $G$ between every sink in $A^\prime$ and some element of $S$.
\end{itemize}
This map is defined in the obvious way: $\psi(A)=(S,A|_{V\setminus S})$, where $S$ is the set of sinks in $A$. Since $f$ is obtained by judiciously iterating $\psi$, it will suffice to show that $\psi$ is a bijection. We omit the simple proof.
\end{proof}
In the language of traces or partially commutative monoids, this theorem is a well known special case of the correspondence between the \emph{Foata normal form} and the \emph{dependence graph} of a given trace (see~\cite{MR1075995}).

As a consequence of this theorem, we have that $(-1)^{|V|-1}\eta_G(-1)$ is the number of acyclic orientations of $G$ with a unique sink at a fixed vertex. This is the result that was mentioned at the beginning of this section.

\section{Path Reciprocity}\label{sec:pathrec}

The purpose of this section is to outline a proof that for any digraph $D$,
\begin{equation}\label{eq:prec}
\omega(\Pi_D)=\Pi_{\overline{D}}.
\end{equation}
Chow gives three proofs of \eqref{eq:prec} in \cite[Corollary 2]{MR1375952}, the first using his reciprocity theorem, the second attributed to Gessel and based on a result of Carlitz, Scoville, and Vaughan \cite[Theorem 7.3]{MR0432472}, and the third using quasisymmetric functions.

The following specialization of \cite[Theorem 7.3]{MR0432472} is the combinatorial essence of \eqref{eq:prec}. Recall that $\pi_D$ is the number of Hamiltonian paths on $D$. The notation $S\uplus T=V$ means $S\cup T=V$ and $S\cap T=\emptyset$.
\begin{prop}\label{prop:pathsum}Let $D$ be a digraph with nonempty vertex set $V$. Then
\begin{equation}\label{eq:pathsum}
\sum_{S\uplus T=V}(-1)^{|T|}\pi_{D|_S}\pi_{\overline{D}|_T}=0.
\end{equation}
\end{prop}
\begin{proof}
Our proof is essentially the one described by Carlitz, Scoville, and Vaughan. Let $H$ be the set of pairs of disjoint paths $(E,E^\prime)$, the first on $D|_S$ and the second on $\overline{D}|_T$, which together cover every vertex. Define a map $i:H\rightarrow H$ as follows.  Let $(E,E^\prime)\in H$. Assume that neither path is empty. Let $v$ be the last vertex of $E$ and $w$ the first vertex of $E^\prime$. Let $e$ be the arrow from $v$ to $w$. If $e\in D$, then let $i(E,E^\prime)$ be the pair of paths obtained by removing $w$ from $E^\prime$ and making it the last vertex of $E$. If $e\notin D$, then let $i(E,E^\prime)$ be the pair of paths obtained by removing $v$ from $E$ and making it the first vertex of $E^\prime$. In the case where $E$ is empty, let $i(E,E^\prime)$ be obtained by removing the first vertex of $E^\prime$ and making it a one-vertex path. If $E^\prime$ is empty, remove the last vertex of $E$.

The map $i$ is clearly a sign-reversing involution, pairing each element of $H$ contributing a positive sign in \eqref{eq:pathsum} with an element contributing a negative sign. Hence the sum is zero.
\end{proof}

Proposition \ref{prop:pathsum} is best understood in terms of set maps. A \emph{set map} is a function with domain the set of finite subsets of a fixed set $V$. For a set map $h$, we use the notation $h_S$ to denote the image of $S\stackrel{,}{\subseteq}V$. When the range is a ring, multiplication of set maps is defined by
\[
(h\cdot g)_U=\sum_{S\uplus T=U}h_S g_T.
\]
So Proposition \ref{prop:pathsum} states that the set maps $h_S\stackrel{,}{=}\pi_{D|_S}$ and $g_S\stackrel{,}{=}(-1)^{|S|}\pi_{\overline{D}|_S}$ are inverse under set map multiplication. See \cite{MR1861053}, \cite{MR1928099}, and \cite{smuccp} for more on set maps in enumerative graph theory.

The following result is due to Lass \cite[Theorem 6.2]{MR1928099}.

\begin{prop}[Lass]\label{prop:mrec}Let $h$ be a set map with multiplicative inverse $h^{-1}$. Then
\[
\omega\left(\sum_{\sigma\vdash S}\tilde{m}_{t(\sigma)}\prod_{T\in\sigma}h_T\right)=
(-1)^{|S|}\sum_{\sigma\vdash S}\tilde{m}_{t(\sigma)}\prod_{T\in\sigma}h^{-1}_T.
\]
\end{prop}
Notice that \eqref{eq:prec} follows from \eqref{eq:pth} and Propositions \ref{prop:pathsum} and \ref{prop:mrec}.




\bibliography{refs}

\def\cprime{$'$}
\providecommand{\bysame}{\leavevmode\hbox to3em{\hrulefill}\thinspace}
\providecommand{\href}[2]{#2}
\begin{thebibliography}{10}

\bibitem{MR0432472}
L.~Carlitz, Richard Scoville, and Theresa Vaughan, \emph{Enumeration of pairs
  of sequences by rises, falls and levels}, Manuscripta Math. \textbf{19}
  (1976), no.~3, 211--243.

\bibitem{MR1375952}
Timothy~Y. Chow, \emph{The path-cycle symmetric function of a digraph}, Adv.
  Math. \textbf{118} (1996), no.~1, 71--98.

\bibitem{MR1358990}
F.~R.~K. Chung and R.~L. Graham, \emph{On the cover polynomial of a digraph},
  J. Combin. Theory Ser. B \textbf{65} (1995), no.~2, 273--290.

\bibitem{MR1075995}
Volker Diekert, \emph{Combinatorics on traces}, Lecture Notes in Computer
  Science, vol. 454, Springer-Verlag, Berlin, 1990.

\bibitem{MR0429577}
Peter Doubilet, \emph{On the foundations of combinatorial theory. {VII}.
  {S}ymmetric functions through the theory of distribution and occupancy},
  Studies in Appl. Math. \textbf{51} (1972), 377--396.

\bibitem{MR675953}
Dominique Foata, \emph{Rearrangements of words}, Combinatorics on words
  (M.~Lothaire, ed.), Encyclopedia of Mathematics and its Applications,
  vol.~17, Addison-Wesley Publishing Co., Reading, Mass., 1983.

\bibitem{MR1778205}
David~D. Gebhard and Bruce~E. Sagan, \emph{Sinks in acyclic orientations of
  graphs}, J. Combin. Theory Ser. B \textbf{80} (2000), no.~1, 130--146,
  \mbox{arXiv:math.CO/9907078}.

\bibitem{MR712251}
Curtis Greene and Thomas Zaslavsky, \emph{On the interpretation of {W}hitney
  numbers through arrangements of hyperplanes, zonotopes, non-{R}adon
  partitions, and orientations of graphs}, Trans. Amer. Math. Soc. \textbf{280}
  (1983), no.~1, 97--126.

\bibitem{MR633783}
Andr{\'e} Joyal, \emph{Une th\'eorie combinatoire des s\'eries formelles}, Adv.
  in Math. \textbf{42} (1981), no.~1, 1--82.

\bibitem{MR1861053}
Bodo Lass, \emph{Orientations acycliques et le polyn\^ome chromatique},
  European J. Combin. \textbf{22} (2001), no.~8, 1101--1123.

\bibitem{MR1928099}
\bysame, \emph{Variations sur le th\`eme {$E+\overline E=XY$}}, Adv. in Appl.
  Math. \textbf{29} (2002), no.~2, 215--242.

\bibitem{MR1354144}
I.~G. Macdonald, \emph{Symmetric functions and {H}all polynomials}, second ed.,
  Oxford Mathematical Monographs, The Clarendon Press Oxford University Press,
  New York, 1995, With contributions by A. Zelevinsky, Oxford Science
  Publications.

\bibitem{math.CO/0209005}
James Propp, \emph{{Lattice structure for orientations of graphs}},
  \mbox{arXiv:math.CO/0209005}.

\bibitem{MR96b:05174}
Richard~P. Stanley, \emph{A symmetric function generalization of the chromatic
  polynomial of a graph}, Adv. Math. \textbf{111} (1995), no.~1, 166--194.

\bibitem{MR1383514}
\bysame, \emph{A matrix for counting paths in acyclic digraphs}, J. Combin.
  Theory Ser. A \textbf{74} (1996), no.~1, 169--172.

\bibitem{MR1207737}
Richard~P. Stanley and John~R. Stembridge, \emph{On immanants of
  {J}acobi-{T}rudi matrices and permutations with restricted position}, J.
  Combin. Theory Ser. A \textbf{62} (1993), no.~2, 261--279.

\bibitem{smuccp}
Gus Wiseman, \emph{Set maps, umbral calculus, and the chromatic polynomial},
  Preprint, \mbox{arXiv:math.CO/0507038}.

\end{thebibliography}

\end{document}